\documentclass{article}
\usepackage{amsmath}
\usepackage{amssymb}
\usepackage{amsbsy}
\usepackage{mathrsfs}
\usepackage{xypic}

\newcommand{\ra}{\rightarrow}

\newcommand{\bb}[1]{\mathbb{#1}}

\newcommand{\Z}{\bb{Z}}
\newcommand{\Q}{\bb{Q}}
\newcommand{\R}{\bb{R}}
\newcommand{\C}{\bb{C}}
\newcommand{\Zp}{\Z_p}
\newcommand{\Zprm}{\Z_\prm}

\newcommand{\Qprm}{\Q_\prm}

\newcommand{\p}{\mathfrak{p}}

\newcommand{\ten}{\otimes}
\newcommand{\teno}[1]{\ten_{#1}}

\newcommand{\gal}[1]{\mathrm{Gal}(#1)}
\newcommand{\of}{\circ}

\newcommand{\br}[1]{\bar{#1}} 

\newcommand{\st}{^\times}
\newcommand{\ie}{i.e. }
\newcommand{\iso}{\cong}
\newcommand{\can}{\simeq}
\newcommand{\Hom}{\mathrm{Hom}}

\newcommand{\End}{\mathrm{End}}
\newcommand{\Aut}{\mathrm{Aut}}
\newcommand{\Det}{\mathrm{det}}
\newcommand{\detf}{\mathrm{D}}
\newcommand{\setm}{\! \smallsetminus \!}

\newcommand{\lam}{\Lambda_m}

\newcommand{\dnd}{\,{{}^{{}_\not}}|\,}

\newcommand{\oh}{\mathcal{O}}
\newcommand{\class}{\mathrm{Cl}}
\newcommand{\gen}[1]{\langle #1 \rangle}

\newcommand{\spc}{\mbox{ }}
\newcommand{\Mod}{\spc\mathrm{mod}\spc}
\newcommand{\con}{\subseteq}

\newcommand{\inv}{\kappa}
\newcommand{\ann}{\mathrm{ann}}
\newcommand{\J}[1]{\mathcal{J}(#1)}
\newcommand{\hJ}[2]{\mathcal{J}^#1(#2)}
\newcommand{\I}{\mathcal{I}}
\newcommand{\E}{\mathcal{E}}
\newcommand{\eps}{\epsilon}

\newcommand{\tors}{\mathrm{tors}}

\newcommand{\qbrc}{\Q^{\mathrm{c}}}
\newcommand{\rf}[1]{\mathcal{R}^{#1}}
\newcommand{\abst}[1]{\mathrm{St}(#1)}

\newcommand{\abKkSw}{\abst{K/k,S,w}}
\newcommand{\abKkS}{\abst{K/k,S}}
\newcommand{\uab}{U_{K/k}^{\mathrm{ab}}}
\newcommand{\sat}{\spc | \spc}

\newcommand{\stick}{\theta}
\newcommand{\halfst}{\tilde{\stick}}

\newcommand{\inpr}[2]{\langle #1,#2 \rangle}
\newcommand{\imp}{\Rightarrow}
\newcommand{\rk}{\mathrm{rk}}

\newcommand{\epsclass}[1]{\br{\eps}(#1)}
\newcommand{\rou}[1]{\mu(#1)}
\newcommand{\prm}{\ell}

\newcommand{\chdot}{_{\textrm{\tiny{$\bullet$}}}}

\newcommand{\tqr}[1]{Q(#1)}
\newcommand{\ndzl}[1]{\tilde{Q}(#1)}
\newcommand{\cyc}{\mathfrak{C}}

\newcommand{\bsch}[2]{\beta_#2^#1}

\newcommand{\stickstick}{\stick_{\mathrm{Stick}}}

\newtheorem{theorem}{Theorem}[section]
\newtheorem{prop}[theorem]{Proposition}
\newtheorem{lemma}[theorem]{Lemma}
\newtheorem{cor}[theorem]{Corollary}
\newtheorem{definition}[theorem]{Definition}
\newtheorem{conj}[theorem]{Conjecture}

\numberwithin{equation}{section}

\newenvironment{remark}
{
\vspace{0.25cm}
\begin{minipage}{11cm}
\begin{small}
\emph{Remark}.
}
{
\end{small}
\end{minipage}
\vspace{0.25cm}
}

\newenvironment{proof}{\emph{Proof}.}{\hspace{\stretch{1}}\rule{1.5ex}{1.5ex} \vspace{5 mm}}

\newenvironment{assumption}{\vspace{5 mm} \textbf{Assumption}. \vspace{5 mm}}

\begin{document}

\begin{center}
\LARGE{The canonical fractional Galois ideal at $s = 0$}
\end{center}

\begin{center}
\textsc{\large{Paul Buckingham}}
\end{center}

\begin{center}
\emph{University of Sheffield, UK}
\end{center}

\vspace{0.5cm}

\begin{abstract}
The Stickelberger elements attached to an abelian extension of number fields conjecturally participate, under certain conditions, in annihilator relations involving higher algebraic $K$-groups. In \cite{snaith:stark}, Snaith introduces canonical Galois modules hoped to appear in annihilator relations generalising and improving those involving Stickelberger elements. In this paper we study the first of these modules, corresponding to the classical Stickelberger element, and prove a connection with the Stark units in a special case.
\end{abstract}

\section{Introduction} \label{intro}

The analytic class number formula is a classical result demonstrating a local-to-global phenomenon in number fields. It relates the residue at $s = 1$ of the locally defined Dedekind $\zeta$-function $\zeta_K(s)$ of a number field $K$ to the globally defined class number $h_K$ of $K$ via a regulator. Reformulated in terms of the leading coefficient $\zeta_K^*(0)$ of $\zeta_K(s)$ at $s = 0$ using the functional equation for $\zeta_K(s)$, it says
\begin{equation} \label{acnf}
\frac{\zeta_K^*(0)}{R_K} = - \frac{h_K}{|\rou{K}|}
\end{equation}
where $R_K$ is the regulator of $K$ and $\rou{K}$ is the group of roots of unity in $K$. As a consequence of (\ref{acnf}), we see that the $\Z$-submodule $\Z \frac{\zeta_K^*(0)}{R_K}$ of $\C$ actually lies in $\Q$ and connects the annihilators of the $\Z$-modules $\rou{K}$ and $\class(K)$:
\[ \ann_\Z(\rou{K}) \Z \frac{\zeta_K^*(0)}{R_K} \con \ann_\Z(\class(K)) .\]

Now, if $k$ is a subfield of $K$ such that $K/k$ is Galois, then $\rou{K}$ and $\class(K)$ come with an action by the group $G = \gal{K/k}$, and one can ask whether it is possible to find non-trivial elements $\alpha \in \Q[G]$ such that
\[ \ann_{\Z[G]}(\rou{K}) \alpha \con \ann_{\Z[G]}(\class(K)) .\]
This more general setting has a classical answer, although it is only partial. It concerns Stickelberger's element, which we now turn to for motivation.

\subsection{The Stickelberger approach}

Let $K/\Q$ be a cyclotomic extension, of conductor $f$ say, and $G$ its Galois group. Stickelberger defined an element
\begin{equation}
\stickstick = \sum_{\stackrel{a=1}{(a,f)=1}}^f \frac{a}{f} (a,K/\Q)^{-1} \in \Q[G] \label{first stick} ,
\end{equation}
and the theorem of 1890 named after him states that any $\Z[G]$-multiple of it having integral coefficients annihilates the class group of $K$. A minor modification of $\stickstick$ can be generalized to arbitrary pairs $(K/k,S)$ where $K/k$ is any abelian\footnote{The word \emph{abelian} may be replaced by the word \emph{Galois} here, but we won't use this and so only give the definition for abelian extensions.} extension of number fields and $S$ a finite set of places of $k$ containing the infinite ones, namely
\[ \stick_{K/k,S}(1) = \sum_{\chi \in \hat{G}} L_{K/k,S}(0,\chi) e_{\br{\chi}} .\]

\begin{remark}
In fact, when $K/\Q$ is cyclotomic and $S$ is the set of places of $\Q$ consisting of the ones ramifying in $K/\Q$, $\stick_{K/\Q,S}(1) = \frac{1}{2}N - \stickstick$ where $N$ is the sum in $\Z[G]$ of the elements of $G = \gal{K/\Q}$. For a proof of this see \cite[p.268]{rosen:number}. Note, however, that the element we call $\stick_{K/k,S}(1)$ is called $\stick_{K/k,S}(0)$ there.
\end{remark}

Now, suppose that $k$ is totally real and $S$ contains the places which ramify in $K/k$. Then it was shown by Klingen that $\stick_{K/k,S}(1) \in \Q[G]$, and by Deligne and Ribet that
\[ \ann_{\Z[G]}(\mu(K)) \stick_{K/k,S}(1) \con \Z[G] ,\]
so in view of Stickelberger's Theorem, it became natural to ask: Do we have
\[ \ann_{\Z[G]}(\mu(K)) \stick_{K/k,S}(1) \con \ann_{\Z[G]}(\class(\oh_K)) \]
in general? This is Brumer's conjecture. Observing that $\mu(K) = \tors(K_1(\oh_K))$ and $\class(\oh_K) = \tors(K_0(\oh_K))$, the question arose of whether there ought to be analogues of Brumer's conjecture involving higher $K$-groups.

\begin{definition} \label{stick}
Let $(K/k,S)$ be as above. Then for an integer $n > 0$, we define the $n$th Stickelberger element to be
\[ \stick_{K/k,S}(n) = \sum_{\chi \in \hat{G}} L_{K/k,S}(1-n,\chi) e_{\br{\chi}} .\]
\end{definition}

\begin{remark}
The element we call $\stick_{K/k,S}(n)$ is called $\stick_{K/k,S}(1 - n)$ in \cite{rosen:number}.
\end{remark}

The following conjecture is a higher-dimensional analogue of the Brumer conjecture, posed when $k$ is totally real, $K$ is totally real or a CM-field, and $S$ contains the places which ramify in $K/k$. See \cite[Ch.7]{snaith:alg}.

\begin{conj} \label{higher brum}
For any $r < 0$, 
\[ \ann_{\Z[G]}(\tors(K_{1-2r}(\oh_{K,S}))) \stick_{K/k,S}(1-r) \con \ann_{\Z[G]}(K_{-2r}(\oh_{K,S})) .\]
\end{conj}

However, in the situation when all the $L$-functions $L_{K/k,S}(s,\chi)$ vanish at $r$, this conjecture becomes trivial. Under the assumption that $K/k$ satisfies the higher Stark conjectures, Snaith shows in \cite{snaith:stark} how one can attach to $K/k$ a family $\{\hJ{r}{K/k}\}_{r \in \Z_{<0}}$ of $\Z[G]$-submodules of $\Q[G]$, hoped to appear in a generalization of Conjecture \ref{higher brum}. Namely,

\begin{conj} \label{gen brum}
If $K/k$ is an abelian extension of number fields with Galois group $G$ and $S$ contains the places which ramify in $K/k$, then for each odd prime $\prm$ and each $r < 0$,
\begin{eqnarray*}
\ann_{\Zprm[G]}(\tors(K_{1-2r}(\oh_{K,S})) \teno{\Z} \Zprm) \hJ{r}{K/k} \cap \Zprm[G] \\
\con \ann_{\Zprm[G]}(K_{-2r}(\oh_{K,S}) \teno{\Z} \Zprm) .
\end{eqnarray*}
\end{conj}

Conjecture \ref{gen brum} is formulated in \cite[Section 5.1]{snaith:stark}, and verified when $k = \Q$ in \cite[Section 6]{snaith:stark} assuming the Quillen--Lichtenbaum conjecture (which relates \'etale cohomology to $K$-theory).

In the present paper, we define an $r = 0$ analogue of the fractional ideals $\hJ{r}{K/k}$. This will be a canonical $\Z[G]$-submodule $\J{K/k,S}$ of $\Q[G]$ whose existence relies upon the truth of the ``Stark conjecture at $0$'', \ie the classical Stark conjecture. In Section \ref{desc of j}, we will describe $\J{K/k,S}$ in terms of \emph{Stark units} (see Section \ref{abKkS section}) under an assumption on the orders of vanishing of the $L$-functions of the extension $K/k$, and work it out explicitly in some special cases (see Section \ref{examples of j}).

The work of Sections \ref{special case} and \ref{examples of j} will be combined with the main result of Section \ref{class-groups}, Theorem \ref{ann theorem}, to prove a relation of the fractional ideal with class-groups. Theorem \ref{ann theorem} connects the annihilator ideals of class-groups and certain quotients of $S$-units. In fact, such a connection has been looked at before. In \cite{cg:fitting} Cornacchia and Greither prove

\begin{theorem} \label{corn-greith}
Let $G = \gal{K/\Q}$ where $K/\Q$ is a totally real abelian extension of prime-power conductor. Then the $\Z[G]$-Fitting ideals of $\oh_K\st/{\cyc_K}$ and $\class(K)$ are equal, where $\cyc_K$ denotes the cyclotomic units in $K$.
\end{theorem}

(See \cite[Section 8.1]{wash:cyc}, \cite[Section 4]{sinnott:stick} or \cite[Section 1]{cg:fitting} for a definition of the cyclotomic units in $K$. Note that they are called \emph{circular units} in \cite{sinnott:stick}.) Recall that the Fitting ideal of a module is always contained in the annihilator ideal, but that equality need not hold. However, as was pointed out by the reviewer, Theorem \ref{corn-greith} implies Theorem \ref{ann theorem} directly. Indeed, since the Pontryagin dual of the particular quotient $U(\prm)/\E^+(\prm)$ appearing in Theorem \ref{ann theorem} is cyclic as a Galois module, and since further the Galois group is cyclic, the Fitting ideal of $U(\prm)/\E^+(\prm)$ is the whole annihilator ideal. $U(\prm)/\E^+(\prm)$ can be shown to be just the $\prm$-part of the units modulo cyclotomic units, and then Theorem \ref{corn-greith} can be applied.

The author is further grateful to the reviewer for bringing to his attention a paper of Rubin (\cite{rubin:global}). Indeed, that paper, which generalizes work of Thaine in \cite{thaine:classgroup}, can also be used to prove Theorem \ref{ann theorem} -- see \cite[Thms 1.3, 2.2]{rubin:global}. This notwithstanding, the theory here is hoped to be applicable in a wider setting and as part of a general principle.

\begin{center}
\textbf{Acknowledgments}
\end{center}

The author would like to thank Victor Snaith for all his help, and the reviewer for many valuable suggestions.

\section{Definition of $\J{K/k,S}$} \label{sec def j}

For the rest of the paper, let us fix once and for all an algebraic closure $\br{\Q}$ of $\Q$ in which our number fields are to lie. $\br{\Q}$ should be thought of as being distinct from the algebraic closure $\qbrc$ which we will choose later for group characters to have their values in.

For the moment we take $K/k$ to be any Galois extension of number fields, although $\J{K/k,S}$ will be defined only for abelian extensions. Again, $S$ will be a finite set of places of $k$ containing the infinite ones, and $S_K$ the set of places of $K$ above those in $S$. We write $\oh_{K,S}$ for $\oh_{K,S_K}$.

\subsection{The Stark regulator} \label{stark reg}

Define $X$ to be the $\Z[G]$-submodule of degree-zero elements in the free abelian group on $S_K$ (with the natural $G$-action). Then we have the Dirichlet regulator map
\begin{eqnarray*}
\lambda : \oh_{K,S}\st \teno{\Z} \R &\ra& X \teno{\Z} \R \\
u &\mapsto& \sum_{w \in S_K} \log \|u\|_w w .
\end{eqnarray*}
This is an isomorphism of $\R[G]$-modules. Now, by the remark immediately after the proof of \cite[Prop.33]{serre:linear} (found in Section 12.1 there), any two $\Q$-representations which become isomorphic over $\C$ must necessarily be isomorphic over $\Q$. Hence there is a $\Q[G]$-module isomorphism
\[ f : \oh_{K,S}\st \teno{\Z} \Q \ra X \teno{\Z} \Q .\]
(We point out, however, that there is in general no canonical choice for $f$.) Then for a finitely generated $\C[G]$-module $V$, let $R_V^f$ denote the determinant of the $\C$-linear map
\begin{eqnarray*}
\Hom_{\C[G]}(V^*,X \teno{\Z} \C) &\ra& \Hom_{\C[G]}(V^*,X \teno{\Z} \C) \\
\phi &\mapsto& \lambda \of f^{-1} \of \phi ,
\end{eqnarray*}
where $V^*$ is the contragredient representation of $V$, and set
\[ \mathcal{R}^f(V) = R_V^f/{L_{K/k,S}^*(0,V)} \]
where $L_{K/k,S}^*(0,V)$ denotes the leading coefficient of the Laurent series of the $L$-function at $s = 0$. If $V$ has character $\chi$, we also write $R_\chi^f = R_V^f$ and $\mathcal{R}^f(\chi) = \mathcal{R}^f(V)$. We call the non-zero complex number $R_\chi^f$ the Stark $f$-regulator for $\chi$.

\subsection{Stark's conjecture} \label{sec stark}

We reproduce here Stark's conjecture as formulated in \cite[Ch.I, Section 5]{tate:stark}. $K/k$ is still an arbitrary Galois extension of number fields with Galois group $G$.

\begin{conj} \label{stark conj}
Let $\chi$ be a (not-necessarily irreducible) character of $G$.

\begin{tabular}{cp{9.5cm}l}
(i) & $\mathcal{R}^f(\chi) \in \Q(\chi)$. \\
(ii) & For every $\sigma \in \gal{\Q(\chi)/\Q}$, $\mathcal{R}^f(\chi)^\sigma = \mathcal{R}^f(\chi^\sigma)$.
\end{tabular}
\end{conj}

\begin{remark}
The truth of Conjecture \ref{stark conj} is shown in \cite[Ch.I, Section 7]{tate:stark} to be independent of the set $S$ and the choice of $\Q[G]$-module isomorphism $f$, and hence is a property solely of the extension. Stark's conjecture is known to hold whenever $K/\Q$ is abelian.
\end{remark}

Assume, now, that $K/k$ satisfies Stark's conjecture and let $\qbrc$ be the algebraic closure of $\Q$ in $\C$. Then we see immediately that

\begin{prop} \label{algebraic}
For any choice of $f : \oh_{K,S}\st \teno{\Z} \Q \ra X \teno{\Z} \Q$, the element $\mathcal{R}^f$ of $\Hom(R(G),\C\st)$ lies in $\Hom_{G_\Q}(R(G),(\qbrc)\st)$, where $G_\Q = \gal{\qbrc/\Q}$ and $R(G)$ is the representation ring of $G$.
\end{prop}

From now on, $K/k$ is assumed to be abelian. Consider the group isomorphism
\begin{eqnarray*}
\varphi_G : \Hom(R(G),(\qbrc)\st) &\ra& \qbrc[G]\st \\
h &\mapsto& \sum_{\chi \in \hat{G}} h(\chi) e_\chi .
\end{eqnarray*}
$\varphi_G$ restricts to give an isomorphism between $\Hom_{G_\Q}(R(G),(\qbrc)\st)$ and $\Q[G]\st$, and so by Proposition \ref{algebraic}, $\varphi_G(\mathcal{R}^f) \in \Q[G]\st$.

Since $X$ naturally embeds into $X \teno{\Z} \Q$, we can define $\I^f$ to be the $\Z[G]$-submodule of $\Q[G]$ generated by
\[ \{\Det_{\Q[G]}(\phi) \sat \phi \in \End_{\Q[G]}(X \teno{\Z} \Q), \phi \of f(\oh_{K,S}\st) \con X \} .\]
(Any $\Q[G]$-module is projective, so $\Det_{\Q[G]} : \End_{\Q[G]}(X \teno{\Z} \Q) \ra \Q[G]$ is a well-defined function.)

\begin{lemma} \label{det lemma}
Let $W$ be a finitely generated $\C[G]$-module, $u \in \Aut_{\C[G]}(W)$, and $V$ a one-dimensional representation of $G$ with character $\chi$. Then the determinant of the $\C$-linear map
\begin{eqnarray*}
\Hom_{\C[G]}(V,W) &\ra& \Hom_{\C[G]}(V,W) \\
\phi &\mapsto& u \of \phi
\end{eqnarray*}
is $\chi(\Det_{\C[G]}(u))$.
\end{lemma}

\begin{proof}
The statement for general $W$ follows immediately from that for free $W$, and in this case the proof is straight-forward linear algebra.
\end{proof}

\begin{prop} \label{ind of f}
For any two choices of $\Q[G]$-module isomorphism
\[ f,g : \oh_{K,S}\st \teno{\Z} \Q \ra X \teno{\Z} \Q ,\]
we have $\I^f \inv(\varphi_G(\mathcal{R}^f)^{-1}) = \I^g \inv(\varphi_G(\mathcal{R}^g)^{-1})$, where $\inv$ is the involution of $\Q[G]$ sending each group element to its inverse.
\end{prop}

\begin{proof}
We use the following useful shorthand: If $A$ is an abelian group and $R$ a subring of $\C$, write $RA$ for $A \teno{\Z} R$. Also, if $V$ and $W$ are finitely generated $\C[G]$-modules and $h$ an endomorphism of $W$, denote by $h_{V,W}$ the endomorphism of $\Hom_{\C[G]}(V,W)$ given by sending a homomorphism $\psi$ to $h \of \psi$.

Let $f,g$ be choices of isomorphism $\Q \oh_{K,S}\st \ra \Q X$ and $u = g \of f^{-1} \in \Aut_{\Q[G]}(\Q X)$. Then for any representation $V$ of $G$
\[ \rf{f}(V) = \Det_\C(u_{V^*,\C X}) \rf{g}(V) ,\]
so if $\rf{(u)}$ is the assignment $V \mapsto \Det_\C(u_{V^*,\C X})$, we have $\rf{f} = \rf{(u)}\rf{g}$ as elements of $\Hom_{G_\Q}(R(G),(\qbrc)\st)$.

Now, if $V$ is a one-dimensional representation with character $\chi$, $\rf{(u)}(\chi) = \Det_\C(u_{V^*,\C X}) = \br{\chi}(\Det_{\Q[G]}(u))$ by Lemma \ref{det lemma}. Hence $\varphi_G(\rf{(u)}) = \inv(\Det_{\Q[G]}(u))$. On the other hand, it is easy to see that $\I^g = \Det_{\Q[G]}(u)^{-1} \I^f$. Putting this together, we have the result.
\end{proof}

Proposition \ref{ind of f} allows us to make the following definition:

\begin{definition} \label{def j}
Define $\J{K/k,S} = \I^f(\inv(\varphi_G(\mathcal{R}^f)^{-1}))$ for any choice of $\Q[G]$-module isomorphism $f : \oh_{K,S}\st \teno{\Z} \Q \ra X \teno{\Z} \Q$.
\end{definition}

\subsection{$\J{K/K,S}$ and the analytic class number formula}

Let us first of all check that when we neglect any Galois action, \ie when $k = K$, the fractional ideal just defined relates the annihilators of $\rou{K}$ and $\class(K)$ in the expected way. So, take $k = K$ and consider the $\Z$-submodule $\J{K/K,S}$ of $\Q$. Then choosing a $\Q$-linear isomorphism $f : \oh_{K,S}\st \teno{\Z} \Q \ra X \teno{\Z} \Q$ sending a $\Z$-basis for $\oh_{K,S}\st/\tors$ to a $\Z$-basis for $X$, we find that the Stark $f$-regulator is just the usual Dirichlet regulator for the pair $(K,S)$, up to sign. Also, in this situation $\I^f$ is simply $\Z$. Then denoting by $S_\infty$ the set of infinite places of $K$ and letting $R_K$ be the Dirichlet regulator for $K$, $\J{K/K,S_\infty}$ is given by
\begin{equation} \label{K/K}
\J{K/K,S_\infty} = \Z \frac{\zeta_K^*(0)}{R_K}
\end{equation}
where $\zeta_K(s)$ is the Dedekind $\zeta$-function of $K$. The analytic class number formula as stated in (\ref{acnf}) then tells us that
\[ \ann_\Z(\rou{K}) \J{K/K,S_\infty} \con \ann_\Z(\class(K)) .\]

\subsection{The integers $r_{K/k,S}(\chi)$}

Because the orders of vanishing of $L$-functions at the point $0$ will be referred to often, we reserve notation for them.

\begin{definition}
If $\chi$ is a (not-necessarily irreducible) character of $G$, then the order of vanishing of the $L$-function $L_{K/k,S}(s,\chi)$ at $s = 0$ is denoted $r_{K/k,S}(\chi)$, or $r(\chi)$ when the pair $(K/k,S)$ is understood.
\end{definition}

We give a lemma which makes an important link between the Galois module structure of $X$ and the integers $r_{K/k,S}(\chi)$. The lemma is in fact \cite[Ch.I,Prop.3.4]{tate:stark}, and does not assume $G$ is abelian.

\begin{lemma} \label{x and l}
For each $v \in S$, choose $w_v | v$. Then for any representation $V$ of $G$ with character $\chi$,
\[ r(\chi) = \sum_{v \in S} \dim_\C V^{G_{w_v}} - \dim_\C V^G = \inpr{\chi}{\chi_X} = \dim_\C \Hom_{\C[G]}(V^*,X \teno{\Z} \C) ,\]
where $\chi_X$ is the character of $X \teno{\Z} \C$.
\end{lemma}

\subsection{$\J{K/k,S}$ and $\stick_{K/k,S}(1)$}

Let $K/k$ be any abelian extension and $S$ any finite set of places of $k$ containing the infinite ones. We show here that $\stick_{K/k,S}(1)$ generates an easily described submodule of $\J{K/k,S}$. Let $e_0 = e_0(K/k,S) \in \C[G]$ be the sum of the idempotents corresponding to the characters $\chi \in \hat{G}$ with $r(\chi) = 0$. In fact, $e_0 \in \Q[G]$: by Lemma \ref{x and l}, it is the $\Q[G]$-determinant of the zero endomorphism of $X \teno{\Z} \Q$.

\begin{prop} \label{stick and j}
$\stick_{K/k,S}(1) \in \J{K/k,S}$. Further,
\[ \Z[G]\stick_{K/k,S}(1) = e_0 \J{K/k,S} .\]
\end{prop}

\begin{proof}
Choose any $\Q[G]$-isomorphism $f : \oh_{K,S}\st \teno{\Z} \Q \ra X \teno{\Z} \Q$, and let $\Psi = \{\chi \in \hat{G} \sat r(\chi) = 0\}$. $e_0 \in \I^f$ since the zero map certainly satisfies the integrality condition in the definition of $\I^f$. Also, for $\chi \in \Psi$, $R_\chi^f = 1$ and $L_{K/k,S}^*(0,\chi) = L_{K/k,S}(0,\chi) = \br{\chi}(\stick_{K/k,S}(1))$, so that
\[ \inv(\varphi_G(\rf{f}))^{-1} = \sum_{\chi \in \hat{G} \setm \Psi} \rf{f}(\br{\chi})^{-1} e_\chi + \stick_{K/k,S}(1) e_0 .\]
But since $L_{K/k,S}(0,\chi) = 0$ for $\chi \in \hat{G} \setm \Psi$, $\stick_{K/k,S}(1)e_0 = \stick_{K/k,S}(1)$. This proves that $\stick_{K/k,S}(1) \in \J{K/k,S}$. To show that $\Z[G] \stick_{K/k,S}(1)$ is all of $e_0 \J{K/k,S}$, use the fact that for any endomorphism $\alpha$ of $X \teno{\Z} \Q$, $e_0 \Det_{\Q[G]}(\alpha) = e_0$.
\end{proof}

When $K/\Q$ is cyclotomic of prime-power conductor and $S$ consists of the ramified places, $e_0$ in Proposition \ref{stick and j} is the minus idempotent $e_- = \frac{1}{2}(1 - c)$ for complex conjugation. For a general cyclotomic field, this need not be the case as there may exist odd characters $\chi$ for which $L_{K/\Q,S}(0,\chi) = 0$.

\section{$\J{K/k,S}$ in a special case} \label{special case}

In this section, we express $\J{K/k,S}$ in terms of the annihilator of a certain quotient of $\oh_{K,S}\st$, under assumptions on the orders of vanishing of the $L$-functions of the extension. The precise statement is in Theorem \ref{j theorem}. This quotient involves the so-called \emph{Stark elements}, which in general exist only conjecturally, but which are known to exist in some cases (in particular the case we will describe).

\subsection{The conjecture $\abKkS$} \label{abKkS section}

We define two $\Z[G]$-submodules of $\oh_{K,S}\st$ which these Stark elements are going to lie in. Let $e$ denote the number of roots of unity in $K$.

\begin{definition}
$\uab = \{u \in \oh_{K,S}\st \sat K(u^{1/e})/k \textrm{ is abelian}\}$.
\end{definition}

\begin{definition} \label{uv}
If $v \in S$ splits completely in $K/k$, we define $U^{(v)}$ in two cases. Namely:

\emph{(a) $\#S \geq 3$:} $U^{(v)} = \{u \in \oh_{K,S}\st \sat \|u\|_{w'} = 1 \textrm{ for all } w' \dnd v\}$.

\emph{(b) $\#S = 2$:}
\[ U^{(v)} = \{u \in \oh_{K,S}\st \sat \|u\|_{w'} = \|u\|_{w''} \textrm{ for all } w',w'' \in S_K \setm \{w|v\}\} .\]
\end{definition}

For future reference, we give names to three hypotheses on the set $S$ necessary for the formulation of the conjecture concerning the Stark elements. They are (St1), (St2) and (St3) as follows:

\begin{tabular}{cp{9.5cm}l}
(St1) & $S$ contains (the infinite places and) the places which ramify in $K/k$. \\
(St2) & $S$ contains at least one place which splits completely in $K/k$. \\
(St3) & $\#S \geq 2$.
\end{tabular}

\begin{remark}
These conditions imply that $r(\chi) \geq 1$ for all $\chi \in \hat{G}$ -- use Lemma \ref{x and l}.
\end{remark}

So, let $S$ satisfy (St1), (St2) and (St3). For a place $w$ of $K$ in $S_K$ having trivial decomposition group in $G$, the following conjecture for the triple $(K/k,S,w)$ will be referred to as $\abKkSw$. (It can be found in \cite[Ch.IV]{tate:stark}.)

\begin{conj} \label{ab stark}
Let $v$ be the place of $k$ below $w$. Then there is $\eps \in \uab \cap U^{(v)}$ such that
\begin{equation} \label{ab conc 1}
\log \|\eps\|_{\sigma w} = -e \zeta_{K/k,S}'(0,\sigma^{-1}) \textrm{ for all } \sigma \in G ,\
\end{equation}
\ie
\begin{equation} \label{ab conc 2}
L_{K/k,S}'(0,\chi) = - \frac{1}{e} \sum_{\sigma \in G} \br{\chi}(\sigma) \log \|\eps\|_{\sigma w} \textrm{ for all } \chi \in \hat{G} ,
\end{equation}
\ie
\begin{equation} \label{ab conc 3}
\sum_{\chi \in \hat{G}} L_{K/k,S}'(0,\chi) e_{\br{\chi}} = - \frac{1}{e} \sum_{\sigma \in G} \log \|\eps\|_{\sigma w} \sigma .
\end{equation}
\end{conj}

If an $\eps$ satisfying $\abKkSw$ exists then it is necessarily unique up to a root of unity in $K$. We therefore see that the triple $(K/k,S,w)$ defines a class in $\oh_{K,S}\st/{\rou{K}}$, where $\rou{K}$ denotes the roots of unity in $K$, and in fact any element in this class satisfies the conjecture.

\begin{definition}
Suppose we have $\eps$ satisfying $\abKkSw$. Then the class $\eps \rou{K}$ will be denoted $\epsclass{K/k,S,w}$, and its elements will be called the \emph{Stark elements} attached to $w$.
\end{definition}

 The conjectures $\abKkSw$, as $w$ runs through all places in $S_K$ having trivial decomposition group, are equivalent, and we call them all just $\abKkS$.

\subsubsection{Stark units}

Assume $\abKkS$ holds and let $w$ be a place of $K$ lying above $S$ and having trivial decomposition group in $G$. Then the $\Z[G]$-submodule of $\oh_{K,S}\st$ generated by $\rou{K}$ and a Stark element for $w$ is independent of the choice of $w$, and we denote it $\E = \E_{K/k,S}$. It will be called the group of $\emph{Stark units}$ for the pair $(K/k,S)$.

\subsection{The assumption on $(K/k,S)$} \label{assump}

We here discuss the assumption to be made on the pair $(K/k,S)$ (extra to the hypotheses (St1), (St2) and (St3) which are required for the formulation of $\abKkS$) in order to state and prove Theorem \ref{j theorem}.

\begin{prop} \label{equiv conds}
Assume (St1), (St2) and (St3) hold, with $v$ splitting completely in $K/k$ and $v' \in S \setm \{v\}$, and let $w|v$ and $w'|v'$. Suppose $\abKkS$ is true and let $\eps \in \epsclass{K/k,S,w}$. Then the following are equivalent:

\begin{tabular}{cp{9.5cm}l}
(i) & $\eps$ generates $\oh_{K,S}\st \teno{\Z} \Q$ freely over $\Q[G]$. \\
(ii) & $\oh_{K,S}\st \teno{\Z} \Q \iso \Q[G]$. \\
(iii) & $w' - w$ generates $X$ freely over $\Z[G]$. \\
(iv) & $X \iso \Z[G]$. \\
(v) & $r(\chi) = 1$ for all $\chi \in \hat{G}$.
\end{tabular}
\end{prop}

\begin{proof}
(i) $\imp$ (ii) $\imp$ (v) and (iii) $\imp$ (iv) $\imp$ (v) are immediate from Lemma \ref{x and l}. To finish the proof, we show (v) $\imp$ (i) and (v) $\imp$ (iii). We do these simultaneously.

Assuming (v), we know first of all (from Lemma \ref{x and l}) that $S$ has two elements, \ie $S = \{v,v'\}$. But we also know from Lemma \ref{x and l} that $X \teno{\Z} \C \iso \C[G]$ and so
\begin{eqnarray*}
|G| &=& \rk(X) \\
&=& \#\{\textrm{places above $v$}\} + \#\{\textrm{places above $v'$}\} - 1 \\
&=& |G| + \#\{\textrm{places above $v'$}\} - 1,
\end{eqnarray*}
hence $w'$ is the unique place of $K$ above $v'$. Thus we see already that $w' - w$ necessarily generates $X$ freely over $\Z[G]$, \ie (iii) holds.

To continue, by definition of the regulator map $\lambda$,
\[ \lambda(\eps) = \sum_{\sigma \in G} \log \|\eps\|_{\sigma w} \sigma w + \log \|\eps\|_{w'} w' .\]
Referring back to the statement of $\abKkSw$ and noting that $L_{K/k,S}'(0,\chi) = L_{K/k,S}^*(0,\chi)$ for all $\chi \in \hat{G}$, we see then that
\[ \lambda(\eps) = -e \sum_{\chi \in \hat{G}} L_{K/k,S}^*(0,\chi) e_{\br{\chi}} w + \log \|\eps\|_{w'} w' ,\]
and so for $\sigma \in G$,
\begin{equation} \label{lam sig eps}
\lambda(\eps^\sigma) = -e \sum_{\chi \in \hat{G}} \br{\chi}(\sigma) L_{K/k,S}^*(0,\chi) e_{\br{\chi}} w + \log \|\eps\|_{w'} w' .
\end{equation}
Now, suppose we have $a_\sigma \in \Q$ for each $\sigma \in G$ such that $\sum_{\sigma \in G} \eps^\sigma \ten a_\sigma = 0$ in $\oh_{K,S}\st \teno{\Z} \Q$. Applying $\lambda$ to both sides, we find using (\ref{lam sig eps}) that $\sum_{\sigma \in G} a_\sigma \chi(\sigma) = 0$ for all $\chi \in \hat{G}$, \ie $a_\sigma = 0$ for all $\sigma \in G$. Since $\rk(\oh_{K,S}\st) = \rk(X) = |G|$, (i) holds.
\end{proof}

Now, the assumption we make on $(K/k,S)$ is:

\begin{assumption}
(St1), (St2) and (St3) hold for $S$, and $r(\chi) = 1$ for all $\chi \in \hat{G}$.
\end{assumption}

In particular $r(1) = 1$ so that $\#S = 2$, and so by \cite[Ch.IV, Prop.3.10]{tate:stark}, $\abKkS$ holds automatically. Hence by Proposition \ref{equiv conds}, $\oh_{K,S}\st \teno{\Z} \Q$ and $X \teno{\Z} \Q$ are free, rank $1$ $\Q[G]$-modules and we have natural choices for free generators.

In looking for examples of pairs $(K/k,S)$ satisfying the above assumption, it is perhaps more convenient to use the following form of the assumption: $S$ contains the infinite and ramified places and equals $\{v,v'\}$ where $v$ splits completely and $v'$ is non-split. We have the following examples: (We point out that if $(K/k,S)$ satisfies the assumption, then so does $(E/k,S)$ for any subextension $E/k$.)

\begin{tabular}{cp{9.5cm}l}
(i) & $p$ an odd prime, $k = \Q$, $K = \Q(\zeta_{p^r})^+$, $S = \{\infty,p\}$. \\
(ii) & $p$ an odd prime, $k = \Q$, $K/\Q$ any finite subextension of the cyclotomic $\Zp$-extension of $\Q$, $S = \{\infty,p\}$. \\
(iii) & $p \equiv 3 \Mod 4$, prime, $k = \Q(\sqrt{-p})$, $K = \Q(\zeta_{p^r})$, $S = \{v,\p\}$ where $v$ is the infinite place of $k$ and $\p$ the unique place above $p$. \\
\end{tabular}

We also remark that if we are already given an extension $K/\Q$ such that $(K/\Q,S)$ satisfies the assumption, where $S = \{\infty,p\}$, then $(KF/F,S_F)$ will also satisfy it for any imaginary quadratic field $F$ such that $p$ remains non-split in $KF$. This will happen in particular whenever $p$ is non-split in $F$ and $K/\Q$ has odd degree (as in example (ii) above).

\subsection{Description of $\J{K/k,S}$} \label{desc of j}

We emphasize that we proceed under the assumption (discussed in Section \ref{assump}) that $S$ satisfies (St1), (St2) and (St3) and $r(\chi) = 1$ for all $\chi \in \hat{G}$. As mentioned, $\abKkS$ holds in this case and we choose $v,v',w,w',\eps$ as in Proposition \ref{equiv conds}.

\begin{theorem} \label{j theorem}
Let $\E$ be the group of Stark units attached to $(K/k,S)$, \ie the $\Z[G]$-submodule of $\oh_{K,S}\st$ generated by $\eps$ and the roots of unity in $K$. Then
\[ \J{K/k,S} = \frac{1}{e} \ann_{\Z[G]}(\oh_{K,S}\st/\E) .\]
\end{theorem}

To prove this, we choose a particular $\Q[G]$-module isomorphism
\[ f : \oh_{K,S}\st \teno{\Z} \Q \ra X \teno{\Z} \Q ,\]
namely the one which sends $\eps$ to $w' - w$ (which exists and is unique by our assumption and Proposition \ref{equiv conds}), and look at $\I^f$ and $\inv(\varphi_G(\mathcal{R}^f))$. This will be done in the following two lemmas.

\begin{lemma} \label{rf lemma}
With $f$ as above, $\varphi_G(\mathcal{R}^f) = e$, the number of roots of unity in $K$.
\end{lemma}

\begin{proof}
From the explicit description of $\lambda(\eps)$ given in the proof of Proposition \ref{equiv conds}, we obtain
\begin{eqnarray*}
\lambda \of f^{-1}(w' - w) &=& \lambda(\eps) \\
&=& -e \sum_{\psi \in \hat{G}} L_{K/k,S}^*(0,\psi) e_{\br{\psi}} w + \log \|\eps\|_{w'} w' \\
&=& e \sum_{\psi \in \hat{G}} L_{K/k,S}^*(0,\psi) e_{\br{\psi}} (w' - w) \\
& & - e \sum_{\psi \in \hat{G}} L_{K/k,S}^*(0,\psi) e_{\br{\psi}} w' + \log \|\eps\|_{w'} w' .
\end{eqnarray*}

Now observe that if $\chi \in \hat{G}$ and $x \in X \teno{\Z} \C$, then $\lambda \of f^{-1}(e_\chi x) = R_{\br{\chi}}^f e_\chi x$. In particular,
\begin{eqnarray*}
R_{\br{\chi}}^f e_\chi (w' - w) &=& \lambda \of f^{-1}(e_\chi(w' - w)) \\
&=& e_\chi\left(e \sum_{\psi \in \hat{G}} L_{K/k,S}^*(0,\psi) e_{\br{\psi}} (w' - w)\right) \\
& & - e_\chi \left(e \sum_{\psi \in \hat{G}} L_{K/k,S}^*(0,\psi) e_{\br{\psi}} w' + \log \|\eps\|_{w'} w'\right) \\
&=& e L_{K/k,S}^*(0,\br{\chi}) e_\chi (w' - w) - (e L_{K/k,S}^*(0,\br{\chi}) - \log \|\eps\|_{w'}) e_\chi w' .
\end{eqnarray*}

However, $e_\chi w' = 0$ for $\chi \in \hat{G} \setm \{1\}$, and from (\ref{ab conc 2}) in $\abKkSw$, $\log \|\eps\|_{w'} = e L_{K/k,S}^*(0,1)$. Therefore $R_\chi^f = e L_{K/k,S}^*(0,\chi)$ for all $\chi \in \hat{G}$.
\end{proof}

\begin{lemma} \label{if lemma}
With $f$ as above, $\I^f = \ann_{\Z[G]}(\oh_{K,S}\st/\E)$.
\end{lemma}

\begin{proof}
Since $X \teno{\Z} \Q$ is free on one generator, $\I^f$ takes the simpler form
\[ \I^f = \{\alpha \in \Q[G] \sat \alpha f(\oh_{K,S}\st) \con X \} .\]
Denote by $U'$ and $\E'$ the images of $\oh_{K,S}\st$ and $\E$ resp. in $\oh_{K,S}\st \teno{\Z} \Q$, and observe that $f$ maps $\E'$ isomorphically onto $X$.

Suppose $\alpha \in \I^f$ and take $u \in \oh_{K,S}\st$. $f(\alpha(u \ten 1)) = \alpha f(u \ten 1) \in X$, and hence $\alpha (u \ten 1) \in \E'$. Therefore $\alpha U' \con \E'$. Conversely, if $\alpha U' \con \E'$ then given $u \in \oh_{K,S}\st$ $\alpha f(u \ten 1) = f(\alpha(u \ten 1)) \in X$. We have so far shown, therefore, that
\[ \I^f = \{\alpha \in \Q[G] \sat \alpha U' \con \E' \} .\]

Now, if $\alpha U' \con \E'$ then in particular $\alpha(\eps \ten 1) \in \E'$, and so $\alpha(\eps \ten 1) = \beta(\eps \ten 1)$ for some $\beta \in \Z[G]$. But then as $\eps \ten 1$ generates $\oh_{K,S}\st \teno{\Z} \Q$ freely over $\Q[G]$, we must have $\alpha = \beta \in \Z[G]$. Hence
\[ \I^f = \ann_{\Z[G]}(U'/{\E'}) = \ann_{\Z[G]}(\oh_{K,S}\st/\E) .\]
\end{proof}

Combining Lemmas \ref{rf lemma} and \ref{if lemma}, we have proved Theorem \ref{j theorem}.

\section{Examples of $\J{K/k,S}$} \label{examples of j}

We describe $\J{K/k,S}$ in three related cases, and compare them with each other. We fix for the whole section the following notation: $p$ is an odd prime, $n$ a positive integer, $\zeta$ a primitive $p^n$th root of unity in $\br{\Q}$ and $K = \Q(\zeta)$. $K^+$ will denote the maximal totally real subfield of $K$, and we have the Galois groups $G = \gal{K/\Q}$ and $G^+ = \gal{K^+/\Q}$. $S$ will be the set $\{\infty,p\}$ of places of $\Q$.

\subsection{$\Q(\zeta_{p^n})^+/\Q$} \label{kplus eps}

The following example is worked out in \cite[Ch.III, Section 5]{tate:stark}. If $w$ is the infinite place of $K^+$ arising from the embedding $\zeta + \zeta^{-1} \mapsto \exp(2\pi i/{p^n}) + \exp(-2\pi i/{p^n})$, then $\epsclass{K^+/\Q,S,w} = \{\pm (1 - \zeta)(1 - \zeta^{-1})\}$. Hence the group $\E^+$ of Stark units in $K$ is generated over $\Z[G^+]$ by $-1$ and $\eps = (1 - \zeta)(1 - \zeta^{-1})$. Because this is an important example, we state Theorem \ref{j theorem} in this special case:

\begin{prop} \label{real ppc j}
$\J{K^+/\Q,S} = \frac{1}{2} \ann_{\Z[G^+]}(\oh_{K^+,S}\st/{\E^+})$.
\end{prop}

We interpret Proposition \ref{real ppc j} in terms of the \emph{cyclotomic units} $\cyc$ of $K^+$. \cite[Lemma 8.1]{wash:cyc} gives the following set of generators for $\cyc$:
\[ \{-1\} \cup \left\{\xi_a = \zeta^{(1-a)/2}\frac{1 - \zeta^a}{1 - \zeta^{\phantom{a}}} \sat 1 < a < \frac{1}{2} p^n , p \dnd a \right\} .\]
The equation
\begin{equation} \label{connect cyclo and stark}
\xi_a^2 = \frac{(1 - \zeta^a)(1 - \zeta^{-a})}{(1 - \zeta)(1 - \zeta^{-1})}
\end{equation}
shows that the cyclotomic units in $K^+$ are closely related to the Stark units.

\begin{definition} \label{up and ep}
Let $U(p) = \oh_{K^+,S}\st \teno{\Z} \Zp$ and $\E^+(p) = \E^+ \teno{\Z} \Zp$.
\end{definition}

\begin{prop} \label{cyclo j int}
\begin{eqnarray*}
\J{K^+/\Q,S} \cap \Z[G^+] &=& \ann_{\Z[G^+]}(\oh_{K^+}\st/\cyc) \\
\textrm{and } \J{K^+/\Q,S} &\con& \ann_{\Zp[G^+]}(U(p)/{\E^+(p)}) .
\end{eqnarray*}
\end{prop}

\begin{proof}
The second part comes directly from Proposition \ref{real ppc j}, and the first part uses (\ref{connect cyclo and stark}).
\end{proof}

\subsection{$\Q(\zeta_{p^n})/\Q$} \label{full cyclo}

Observe that the orders of vanishing of the $L$-functions $L_{K/\Q,S}(s,\psi)$ are given by
\[ r(\psi) = \left\{
\begin{array}{ll}
1 & \textrm{if $\psi \in \hat{G}$ is even} \\
0 & \textrm{if $\psi \in \hat{G}$ is odd.}
\end{array} \right. \]
This shows that $X \teno{\Z} \Q \cong e_+\Q[G]$ where $e_+ = \frac{1}{2}(1 + c)$ is the $+$-idempotent for complex conjugation $c \in G$. We will also use the notation $e_- = \frac{1}{2}(1 - c)$.

Let $w$ be the place of $K$ arising from the embedding $\zeta \mapsto \exp(2\pi i/{p^n})$, and $w'$ the unique place of $K$ above $p$. Then there is a $\Q[G]$-module homomorphism $f : \oh_{K,S}\st \teno{\Z} \Q \ra X \teno{\Z} \Q$ such that
\begin{equation} \label{def can f}
f(1 - \zeta) = w' - w ,
\end{equation}
and it is necessarily unique. Furthermore, $f$ is an isomorphism.

\begin{lemma} \label{full rf lemma}
With $f$ as in (\ref{def can f}),
\[ \inv(\varphi_G(\rf{f}))^{-1} = \frac{1}{2} e_+ + \stick_{K/\Q,S}(1) \left(= \frac{1}{2} e_+ + \stick_{K/\Q,S}(1)e_- \right) .\]
\end{lemma}

\begin{proof}
This is little more than a combination of the techniques found in Proposition \ref{stick and j} and Lemma \ref{rf lemma}.
\end{proof}

The following lemma gives a nicer form for $\I^f$, with $f$ as above.

\begin{lemma} \label{full if lemma}
With $f$ as in (\ref{def can f}), $\I^f$ is the $\Z[G]$-submodule of $\Q[G]$ generated by $\{ \alpha e_+ + e_- \sat \alpha \in \ann_{\Z[G]}(\oh_{K,S}\st/{\E}) \}$, where $\E$ is the $\Z[G]$-submodule of $\oh_{K,S}\st$ generated by $1 - \zeta$.
\end{lemma}

\begin{prop} \label{full j desc}
$\J{K/\Q,S}$ is the $\Z[G]$-submodule of $\Q[G]$ generated by
\[ \left\{ \frac{1}{2} \alpha e_+ + \stick_{K/\Q,S}(1) \sat \alpha \in \ann_{\Z[G]}(\oh_{K,S}\st/{\E}) \right\} .\]
\end{prop}

\begin{proof}
Use Lemmas \ref{full rf lemma} and \ref{full if lemma}.
\end{proof}

\subsubsection{Comparison of $\J{K/\Q,S}$ and $\J{K^+/\Q,S}$.} \label{sec:quotient naturality}

Using the descriptions of $\J{K/\Q,S}$ and $\J{K^+/\Q,S}$ that we found in Propositions \ref{real ppc j} and \ref{full j desc} resp., we are able to give an example of the naturality of the fractional ideal under passing to quotients.

\begin{prop} \label{quotient naturality}
With notation as above, $\J{K^+/\Q,S}$ is the image of $\J{K/\Q,S}$ under the natural map $\Q[G] \ra \Q[G^+]$.
\end{prop}

\subsection{$\Q(\zeta_{p^n})/\Q(\sqrt{-p})$, $p \equiv 3 \Mod 4$} \label{p cong 3 mod 4}

Assume $p$ is a prime congruent to $3 \Mod 4$, so that $K = \Q(\zeta)$ contains the imaginary quadratic field $k = \Q(\sqrt{-p})$, and let $H = \gal{K/k}$. We let $S_k$ be the set of places of $k$ lying above those in $S$. Of course, $S_k$ consists exactly of the infinite place of $k$ and the unique place $\p$ above $p$. Let $w$ be the infinite place of $K$ arising from the embedding $\zeta \mapsto \exp(2\pi i/{p^n})$, and $w_+$ its restriction to the maximal real subfield $K^+$ of $K$.

\begin{definition} \label{half stick}
Define the element $\halfst_{K/\Q,S} \in \Q[H]$ by
\[ \halfst_{K/\Q,S} = \sum_{\sigma \in H} \zeta_{K/\Q,S}(0,\sigma) \sigma^{-1} .\]
\end{definition}

This ``half Stickelberger element'' is obtained from the usual Stickelberger element $\stick_{K/\Q,S}(1)$ by keeping only those terms corresponding to elements of the index two subgroup $H$ of $G$.

\begin{prop} \label{halfst is stark}
Let $\halfst = \halfst_{K/\Q,S}$ be as in Definition \ref{half stick}, and $e$ the number of roots of unity in $K$. Then $(1 - \zeta)^{e\halfst}$ is a Stark element for the triple $(K/k,S_k,w)$.
\end{prop}

\begin{proof}
We show that $(1 - \zeta)^{e\halfst}$ satisfies (\ref{ab conc 2}). (This is sufficient because $\p$ is totally ramified in $K/k$, so that $U^{(\infty)}$ is all of $\oh_{K,S}\st$.) So, take $\chi \in \hat{H}$ and let $\chi_1$ be the corresponding character of $G^+$ and $\chi_2$ the non-trivial extension of $\chi$ to $G$. Then by Frobenius reciprocity together with the inflation/induction properties of $L$-functions,
\begin{equation} \label{frob rec l function prod}
L_{K/k,S_k}'(0,\chi) = L_{K^+/\Q,S}'(0,\chi_1) L_{K/\Q,S}(0,\chi_2) .
\end{equation}
By what we know about Stark units in the extension $K^+/\Q$ (recall the beginning of Section \ref{kplus eps}),
\[ L_{K^+/\Q,S}'(0,\chi_1) = -\frac{1}{2} \sum_{\sigma \in H} \br{\chi}(\sigma) \log \|1 - \zeta\|_{\sigma w} .\]
On the other hand, $L_{K/\Q,S}(0,\chi_2)$ is just equal to $2\sum_{\sigma \in H} \zeta_{K/\Q,S}(0,\sigma)\chi(\sigma)$. Substituting these expressions into (\ref{frob rec l function prod}) gives
\[ L_{K/k,S_k}'(0,\chi) = -\frac{1}{e} \sum_{\sigma \in H} \br{\chi}(\sigma) \log \|(1 - \zeta)^{e\halfst}\|_{\sigma w} ,\]
which is what we wanted.
\end{proof}

Using Theorem \ref{j theorem} and Proposition \ref{halfst is stark}, we now have:

\begin{prop}
$\J{K/k,S_k} = \frac{1}{e} \ann_{\Z[H]}(\oh_{K,S}\st/{\tilde{\E}})$, where $\tilde{\E}$ is the $\Z[H]$-submodule of $\oh_{K,S}\st$ generated by $\zeta$ and $(1 - \zeta)^{e\halfst}$.
\end{prop}

\subsection{Comparison of $\J{K/k,S_k}$ and $\J{K^+/\Q,S}$}

We continue with the notation of Section \ref{p cong 3 mod 4}, and emphasize that $p \equiv 3 \Mod 4$. Since restriction $H \ra G^+$ identifies the Galois groups $H$ and $G^+$, we can consider $\Q[G^+]$ as a $\Z[H]$-submodule of $\Q[H]$. Therefore we can think of $\J{K^+/\Q,S}$ as lying inside $\Q[H]$. Let $\E^+$ be the $\Z[H]$-submodule of $\oh_{K^+,S}\st$ generated by $-1$ and $\eps = (1 - \zeta)(1 - \zeta^{-1})$, and $\tilde{\E}$ the $\Z[H]$-submodule of $\oh_{K,S}\st$ generated by $\zeta$ and $(1 - \zeta)^{e\halfst}$.

\begin{prop} \label{j rel}
\begin{equation} \label{j rel eq 1}
\J{K/k,S_k} = 2\halfst \J{K^+/\Q,S} ,
\end{equation}
where $\halfst = \halfst_{K/\Q,S}$ is the ``half Stickelberger element'' of Section \ref{p cong 3 mod 4}. Equivalently,
\begin{equation} \label{j rel eq 2}
\ann_{\Z[H]}(\oh_{K,S}\st/\tilde{\E}) = e\halfst \ann_{\Z[H]}(\oh_{K^+,S}\st/{\E^+}) ,
\end{equation}
where $e$ is again the number of roots of unity in $K$.
\end{prop}

\begin{proof}
\emph{(Proposition \ref{j rel})} The equivalence of (\ref{j rel eq 1}) and (\ref{j rel eq 2}) is just Theorem \ref{j theorem}. The inclusion ``$\supseteq$'' in (\ref{j rel eq 2}) is almost immediate when one recalls that $\E^+/\tors$ is generated by $(1 - \zeta)(1 - \zeta^{-1})$ while $\tilde{\E}/\tors$ is generated by $(1 - \zeta)^{e\halfst}$. The other inclusion is obtained by observing that $\halfst$ is invertible.
\end{proof}

\subsection{Comparison of $\J{K/\Q,S}$ and $\J{K/k,S_k}$}

We again assume $p \equiv 3 \Mod 4$, and continue with the above notation. The above work allows us to provide an example of the naturality of the fractional ideal under passing to subgroups. It is akin to the base change for Stickelberger elements described in \cite{sands:basechange}. Indeed let $\bsch{G}{H}$ be the element $\beta(0)$ appearing in \cite[Prop.1]{sands:basechange}. We find that in our case, $\bsch{G}{H} = (1 + c)\halfst$, and then Propositions \ref{quotient naturality} and \ref{j rel} easily give

\begin{prop}
Let $\pi_H : \Q[G] \ra \Q[H]$ be the ring homomorphism obtained by extending linearly the projection $G \ra H$ arising from the decomposition $G = H\gen{c}$. Then
\[ \J{K/k,S_k} = \pi_H(\bsch{G}{H} \J{K/\Q,S}) .\]
\end{prop}

\section{Connection with class-groups} \label{class-groups}

As mentioned in Section \ref{intro}, Snaith constructs (in \cite{snaith:stark}) $\Z[G]$-submodules $\hJ{r}{K/k}$ of $\Q[G]$ for abelian extensions $K/k$ satisfying the Stark conjecture at $r < 0$. In \cite{snaith:rel}, however, he had already constructed $\hJ{r}{K/\Q}$ when $K$ is the maximal totally real subfield of a cyclotomic extension of $\Q$. Further, in the same paper he showed (\cite[Theorem 1.8]{snaith:rel}) that the intersection of $\hJ{r}{K/\Q}$ with the $\prm$-adic group-ring $\Zprm[G]$ ($\prm$ an odd prime) lies in the annihilator of a certain \'etale cohomology group when $r<0$ is even. The methods used in the proof of (\cite[Theorem 1.8]{snaith:rel}) work in the setting of the fractional ideal $\J{K/k,S}$ defined in Section \ref{sec def j} when $k = \Q$ and $K$ is the maximal totally real subfield of a cyclotomic extension of $\Q$ having \emph{$\prm$-power conductor}. In this case, the \'etale cohomology becomes the $\prm$-part of the class-group of $K$. This section will give an idea of the methods employed in the proof of \cite[Theorem 1.8]{snaith:rel}, but with emphasis on the $r = 0$ setting that we need.

\subsection{Perfect complexes} \label{perfect cmxs}

If $R$ is a ring, then a chain complex of $R$-modules is said to be \emph{perfect} if it is bounded and the modules making up the complex are finitely generated and projective. Suppose we have an abelian group $G$ and a prime $\prm$, and that we have a perfect complex $F\chdot$ of $\Zprm[G]$-modules, all of whose homology groups are finite. Then we can form an isomorphism
\begin{equation} \label{complex iso}
\bigoplus_j F_{2j} \teno{\Zprm} \Qprm \ra \bigoplus_j F_{2j+1} \teno{\Zprm} \Qprm
\end{equation}
using the exact sequences
\begin{equation} \label{exact sequence 1}
0 \ra B_i(F\chdot) \ra Z_i(F\chdot) \ra H_i(F\chdot) \ra 0
\end{equation}
and
\begin{equation} \label{exact sequence 2}
0 \ra Z_{i+1}(F\chdot) \ra F_{i+1} \ra B_i(F\chdot) \ra 0 .
\end{equation}
One uses that (\ref{exact sequence 2}) splits after tensoring with $\Qprm$, and (\ref{exact sequence 1}) gives isomorphisms $B_i(F\chdot) \teno{\Zprm} \Qprm \iso Z_i(F\chdot) \teno{\Zprm} \Qprm$. For details of how the isomorphism is constructed, see \cite[Section 2]{snaith:rel}.

The determinant of the isomorphism in (\ref{complex iso}) is known to be well-defined up to $\Zprm[G]\st$, \ie if different splittings are chosen then the determinant will change by an element of $\Zprm[G]\st$ (see \cite[Ch.15]{swan:alg}). We denote the class of the isomorphism in $\Qprm[G]\st/{\Zprm[G]\st}$ by $\Det(F\chdot)$.

In fact, this works more generally. $\Zprm[G]$ and $\Qprm[G]$ can be replaced by any commutative rings $R \con S$, and $F\chdot$ by a perfect complex of $R$-modules which becomes exact on tensoring with $S$, though we must now assume that the projectives in $F\chdot$ are in fact free. Then in the same way we obtain an element $\Det(F\chdot) \in S\st/{R\st}$. Further, the $R$-submodule of $S$ that $\Det(F\chdot)^{-1}$ generates (note the inverse) is equal to $\detf_R(F\chdot)$, where $\detf_R$ is the determinant functor introduced in \cite{km:proj}.

The following proposition is \cite[Cor. 2.11]{snaith:rel}.

\begin{prop} \label{ann prop}
Let $G$ be a finite abelian group, $\prm$ a prime, and suppose that $F\chdot$ is a perfect complex of $\Zprm[G]$-modules with finite homology in degrees $0$ and $1$ and zero homology elsewhere. Then if $\Hom(H_1(F\chdot),\Qprm/\Zprm)$ is cyclic as a $\Zprm[G]$-module, we have the containment
\[ \det(F\chdot)^{-1} \ann_{\Zprm[G]}(H_1(F\chdot)) \con \ann_{\Zprm[G]}(H_0(F\chdot)) .\]
\end{prop}

This is a special case of a more general theorem, namely \cite[Theorem 2.4]{snaith:rel}. Proposition \ref{ann prop} as stated is sufficient for our purposes.

\subsection{The complex $C_m$}

In \cite[Section 5]{bg:equiv}, Burns and Greither construct the complex $C_m$ of $\lam$-modules. We will not review the construction here, but suffice it to say that it arises as the mapping cone of a map into an \'etale complex, and its cohomology groups are not difficult to describe. Let $U_m^\infty$ be the inverse limit (with respect to the norm maps) of the modules $\oh_{L_m^n,S}\st \teno{\Z} \Zprm$, and let $\eps_m \in U_m^\infty$ be Soul\'e's cyclotomic element:
\[ \eps_m = ((1-\zeta_m^{\prm^{-n}}\zeta_{\prm^{n+1}})(1-\zeta_m^{-\prm^{-n}}\zeta_{\prm^{n+1}}^{-1}))_n .\]
We also let $X_m^n$ be the kernel of the degree map on the free $\Zprm$-module on the set of \emph{finite} places of $L_m^n$ above those in $S$. Restriction of places gives maps $X_m^n \ra X_m^{n-1}$, and we denote the projective limit by $X_m^\infty$.

\begin{prop} \label{coh of cm}
The complex $C_m$ of $\lam$-modules is acyclic outside degrees $1$ and $2$, and we have:

(i) $H^1(C_m) \can U_m^\infty/{\mathcal{E}_m}$ where $\mathcal{E}_m$ is the $\lam$-submodule of $U_m^\infty$ generated by $\eps_m$.

(ii) There is a canonical short exact sequence
\[ 0 \ra \class_\prm^\infty \ra H^2(C_m) \ra X_m^\infty \ra 0 .\]
\end{prop}

For the proof, see \cite[Section 5.1]{bg:equiv}.

Using Proposition \ref{coh of cm} one deduces that $C_m$ becomes exact after tensoring with the total quotient ring $Q(\lam)$ of $\lam$. Hence, by the discussion in Section \ref{perfect cmxs}, we can take its determinant $\Det(C_m) \in Q(\lam)\st/\lam\st$, and by \cite[Theorem 6.1]{bg:equiv},
\begin{equation} \label{bg det basis}
\lam \Det(C_m)^{-1} = \lam(e_+ + e_- g_m)
\end{equation}
where $g_m$ is the limit of the elements $g_m^n = -\stick_{L_m^n,S}(1)$. (By \cite[(26)]{bg:equiv}, $g_m$ does indeed lie in $Q(\lam)$.)

\subsection{Working at the finite level}

As explained in \cite[p.563]{snaith:rel}, after modifying $C_m$ slightly if necessary we may assume that it is a bounded complex of finitely generated free modules. Consider now the ``finite-level'' complex $C_m^n = C_m \teno{\lam} \Zprm[G_m^n]$. Using \cite[pp.573-575]{snaith:rel}, one finds that the cohomology groups are described as follows: $C_m^n$ is acyclic outside degrees $1$ and $2$, and there are exact sequences

\begin{equation}
0 \ra \E_m^n \ra \oh_{L_m^n,S}\st \teno{\Z} \Zprm \ra H^1(C_m^n) \ra 0
\end{equation}

\begin{equation}
0 \ra \class(\oh_{L_m^n,S}) \teno{\Z} \Zprm \ra H^2(C_m^n) \ra X_m^n \ra 0 ,
\end{equation}
where $\E_m^n$ is the $\Zprm[G_m^n]$-submodule of $\oh_{L_m^n,S}\st \teno{\Z} \Zprm$ generated by $(1 - \zeta_{mp^{n+1}})(1 - \zeta_{mp^{n+1}}^{-1})$.

Now let us take $m = 1$. In this case $X_1^n = 0$ so that $H^2(C_1^n)$ is finite, and by \cite[Ch.8]{wash:cyc} the submodule $\E_m^n$ of $\oh_{L_1^n,S}\st \teno{\Z} \Zprm$ has finite index. Therefore in order to apply Proposition \ref{ann prop}, it would remain to check that $\Hom(H^1(C_1^n),\Qprm/\Zprm)$ is cyclic as a $\Zprm[G_1^n]$-module. In fact, as explained in \cite[pp.575,576]{snaith:rel}, this is the case if we replace $C_1^n$ by $C_1^{n,+}$, the complex obtained by taking plus-parts for complex conjugation. Proposition \ref{ann prop} then says
\begin{equation} \label{ann rel 1}
\Det(C_1^{n,+})^{-1} \ann_{\Zprm[G_1^{n,+}]}(U^{n,+}/{\E^{n,+}}) \con \ann_{\Zprm[G_1^{n,+}]}(\class(\oh_{L_1^{n,+},S}) \teno{\Z} \Zprm)
\end{equation}
where $U^{n,+} = \oh_{L_1^{n,+},S}\st \teno{\Z} \Zprm$ and $\E^{n,+}$ is the $\Zprm[G_1^{n,+}]$-submodule of $U^{n,+}$ generated by $(1 - \zeta_{\prm^{n+1}})(1 - \zeta_{\prm^{n+1}}^{-1})$.

If we had a natural map
\begin{equation} \label{non-existent map}
\tqr{\Lambda_1}\st/{\Lambda_1\st} \ra \Qprm[G_1^n]\st/{\Zprm[G_1^n]\st} ,
\end{equation}
then by the naturality of the construction of $\Det(C_1)$ and $\Det(C_1^n)$, the latter would be the image of the former under this map. However, the projection $\Lambda_1 \ra \Zprm[G_1^n]$ does not pass to total quotient rings and so the map in (\ref{non-existent map}) does not exist.

\subsubsection{Never-divisors-of-zero}

This problem is solved by considering the set of so-called \emph{never-divisors-of-zero} in $\Lambda_1$. This is defined to be the multiplicative subset of $\Lambda_1$ consisting of all those elements in $\Lambda_1$ whose image in $\Zprm[G_1^n]$ is a non-zero-divisor for all $n$. We let $\ndzl{\Lambda_1}$ be the localization of $\Lambda_1$ at the never-divisors-of-zero, and observe that $\ndzl{\Lambda_1}$ is a subring of $\tqr{\Lambda_1}$.

Now, using \cite[Prop.4.5]{snaith:rel}, \cite[Prop.4.4]{rw:lrnc} and the discussion in \cite[p.561]{snaith:rel}, one finds that $C_m \teno{\Lambda_1} \ndzl{\Lambda_1}$ is exact. The naturality of the $\det$ construction in Section \ref{perfect cmxs} then shows that $\det(C_1^n)$ is the image of $\det(C_m)$ under $\ndzl{\Lambda_1}\st/{\Lambda_1\st} \ra \Qprm[G_1^n]\st/{\Zprm[G_1^n]\st}$. Referring back to (\ref{bg det basis}), we therefore see that
\begin{equation} \label{finite level det}
\det(C_1^n)^{-1} = e_+ + e_- g_1^n \Mod \Zprm[G_1^n]\st ,
\end{equation}
where (by abuse of notation) $e_+$ and $e_-$ refer to idempotents for complex conjugation in $\Zprm[G_1^n]$.

\subsection{The annihilator theorem} \label{sec:ann thm}

Consider again the complex $C_1^{n,+}$ obtained by taking invariants under complex conjugation. (\ref{finite level det}) shows that the determinant $\det(C_1^{n,+})$ of $C_1^{n,+}$ is equal to $1$. Since the unique prime of $L_1^n$ above $\prm$ is principal (generated by $(1 - \zeta_{\prm^{n+1}})(1 - \zeta_{\prm^{n+1}}^{-1})$), \cite[Prop.11.6]{neukirch:alg} gives
\[ \ann_{\Zprm[G_1^n]}(\class(\oh_{L_1^n,S}) \teno{\Z} \Zprm) = \ann_{\Zprm[G_1^n]}(\class(L_1^n) \teno{\Z} \Zprm) ,\]
so (\ref{ann rel 1}) becomes
\[ \ann_{\Zprm[G_1^{n,+}]}(U^{n,+}/{\E^{n,+}}) \con \ann_{\Zprm[G_1^{n,+}]}(\class(L_1^{n,+}) \teno{\Z} \Zprm) .\]
Let us state this in tidier notation. We emphasize that we have fixed an odd prime $\prm$.

\begin{theorem} \label{ann theorem}
Let $K^+$ be the maximal totally real subfield of the cyclotomic field $K = \Q(\zeta_{\prm^{n+1}})$ for some $n \geq 0$, $G^+ = \gal{K^+/\Q}$, and $S$ the set $\{\infty,\prm\}$ of places of $\Q$. If $\E^+(\prm)$ is the $\Zprm[G^+]$-submodule of $U(\prm) = \oh_{K^+,S}\st \teno{\Z} \Zprm$ generated by $(1 - \zeta_{\prm^{n+1}})(1 - \zeta_{\prm^{n+1}}^{-1})$, then
\[ \ann_{\Zprm[G^+]}(U(\prm)/{\E^+(\prm)}) \con \ann_{\Zprm[G^+]}(\class(K^+) \teno{\Z} \Zprm) .\]
\end{theorem}

\begin{remark}
Although Vandiver's conjecture predicts the $\prm$-part of  $\class(K^+)$ to be trivial, the techniques used here (as mentioned in Section \ref{intro}) are hoped to be applicable to more general fields.
\end{remark}

It is now an easy consequence, using Proposition \ref{cyclo j int}, that the following holds (adopting the notation $A_\prm = A \teno{\Z} \Zprm$ for any abelian group $A$):

\begin{cor} \label{cyc class group}
Let $K^+$ and $S$ be as in Theorem \ref{ann theorem}. Then
\[ \J{K^+/\Q,S} \con \ann_{\Zprm[G^+]}(\class(K^+)_\prm) .\]
\end{cor}

We can also use Theorem \ref{ann theorem}, together with Stickelberger's Theorem, to prove

\begin{cor}
Let $K = \Q(\zeta_{\prm^{n+1}})$ and $S$ as in Theorem \ref{ann theorem}. Then
\[ \ann_{\Zprm[G]}(\rou{K}_\prm) \J{K/\Q,S} \con \ann_{\Zprm[G]}(\class(K)_\prm) .\]
\end{cor}

\begin{proof}
An element of the left-hand side looks like $\beta(\frac{1}{2}\alpha e_+ + \stick)$ where $\beta \in \ann_{\Zprm[G]}(\rou{K})$, $\alpha \in \ann_{\Z[G]}(\oh_{K,S}\st/\E_{K/\Q,S})$, and $\stick$ is the Stickelberger element $\stick_{K/\Q,S}(1)$. However, $\beta \stick$ annihilates $\class(K)_\prm$ so we consider simply $\frac{1}{2}\beta \alpha e_+$, or, what is more, just $\alpha e_+$.

Now, $\class(K)_\prm = \class(K)_\prm^+ \oplus \class(K)_\prm^-$, and we need only worry about the plus part. But this is canonically isomorphic to $\class(K^+)_\prm$, and then Theorem \ref{ann theorem} gives us what we need.
\end{proof}

\end{document}